\documentclass{article}

\usepackage[utf8]{inputenc}
\usepackage[english]{babel}
\usepackage{array}
\usepackage{amsmath, amsthm, amssymb}
\usepackage{graphicx}
\usepackage{mathtools}
\usepackage[usenames,dvipsnames]{xcolor}
\usepackage{xcolor}
\usepackage{setspace}
\usepackage{babel}
\usepackage{centernot}
\usepackage{bbm}
\usepackage{mathrsfs}
\usepackage{wasysym}
\usepackage[utf8]{inputenc}
\usepackage[a4paper, left=3.5cm, right=3.5cm, top=2cm]{geometry}
\usepackage[backend=biber, style=alphabetic]{biblatex}
\usepackage[framemethod=tikz]{mdframed}

\newmdtheoremenv{lemma}{Lemma}
\newmdtheoremenv{satz}[lemma]{Satz}
\newmdtheoremenv{thm}[lemma]{Theorem}
\newmdtheoremenv{kor}[lemma]{Korollar}
\newmdtheoremenv{cor}[lemma]{Corollary}
\newmdtheoremenv{bsp}[lemma]{Beispiel}
\newmdtheoremenv{exa}[lemma]{Example}
\newmdtheoremenv{prop}[lemma]{Proposition}

\newtheorem*{A}{Lemma A}
\surroundwithmdframed[backgroundcolor=white]{A}
\newtheorem*{B}{Lemma B}
\surroundwithmdframed[backgroundcolor=white]{B}
\newtheorem*{C}{Lemma C}
\surroundwithmdframed[backgroundcolor=white]{C}

\surroundwithmdframed[backgroundcolor=white]{D}
\newtheorem*{1b}{Theorem 1b}
\surroundwithmdframed[backgroundcolor=white]{1b}

\theoremstyle{remark}

\theoremstyle{definition}

\definecolor{champagne}{rgb}{0.97, 0.91, 0.81}
\definecolor{apricot}{rgb}{1, 0.93, 0.89}
\definecolor{bubbles}{rgb}{0.91, 1.0, 1.0}
\definecolor{lavenderblush}{rgb}{1.0, 0.94, 0.96}
\definecolor{pastelyellow}{rgb}{0.93, 0.99, 0.82}
\definecolor{sun}{rgb}{0.99, 0.97, 0.74}

\begin{document}
\title{The bunkbed conjecture still holds for cactus graphs\\
and for graphs with certain biconnected components}
\author{Robin Denart\footnote{Universität Paderborn, email: robin@denart.de}}
\date{June 2025}
\maketitle
\begin{abstract}
Recently, the bunkbed conjecture has been shown to be false, which naturally prompts questions on how to classify the graphs that still satisfy the conjecture. We distinguish between a weak version of the bunkbed conjecture where all the horizontal edges of the bunkbed graph are present with the same probability and a strong version of the conjecture where the edge weights on the underlying graph may be assigned individually. We show that any given graph satisfies either version of the conjecture if and only if all of its biconnected components do. Moreover, we show that all cactus graphs satisfy the strong version, and by combining previous results of other authors, any graph \(G\) such that every biconnected component of \(G\) is either a cycle, complete, complete bipartite, symmetric complete \(k\)-partite or an edge difference of a complete graph and a complete subgraph satisfies the weak version. Furthermore, we apply the aforementioned results to show that any counterexample to the strong version of the bunkbed conjecture contains a non-trivial subdivision of the diamond graph as a minor and demonstrate how this result might be strengthened in the future.
\end{abstract}\vspace{20pt}
\section{Introduction and main results}
Throughout this paper, the word "graph"\text{ }refers to simple graphs \(G = (V,E)\) consisting of a finite vertex set \(V(G) \vcentcolon= V\) and an edge set \(E(G) \vcentcolon= E\). Formally, we view an edge between two vertices \(v,w\in V(G)\) as a set containing exactly those two vertices, denoted by \(vw \vcentcolon= \{v,w\}\). Given two vertices \(v,w\in V(G)\), we let \(v \overset{G}{\longleftrightarrow} w\) denote the statement that there exists a path between \(v\) and \(w\) in \(G\). In this case, there obviously exists such a path that is self-avoiding. Similarly, for \(U \subseteq V(G)\), we write \text{ }\(U \overset{G}{\longleftrightarrow} w \text{ }\vcentcolon \Leftrightarrow \text{ }\exists u\in U\): \(u \overset{G}{\longleftrightarrow} w\).\vspace{6pt}\\
Given a graph \(G\), a vertex \(v\in V(G)\), an edge \(e = xy\in E(G)\) and \(i \in \{0,1\}\), we let \(v_i \vcentcolon= (v,i)\) and \(e_i \vcentcolon= x_iy_i\). We are now interested in the \emph{bunkbed graph} \(B(G)\) of \(G\), consisting of the vertex set \(V(B(G)) \vcentcolon= V(G) \times \{0,1\} = \{v_i \vert v\in V(G), i\in\{0,1\}\}\) and the edge set \(E(B(G)) \vcentcolon= \{e_i \vert e\in E(G), i\in\{0,1\}\} \cup \{v_0v_1 \vert v\in V(G)\}\). The edges \(e_i\) are called \emph{horizontal edges} of \(B(G)\) and the edges \(v_0v_1\) are called \emph{vertical edges} or \emph{posts} of \(B(G)\).\vspace{6pt}\\
Intuitively, one might consider vertices \(v_0, w_1 \in V(B(G))\) from different bunks to be "further apart" than \(v_0\) and \(w_0\) (and as far as the distance metric on connected graphs is concerned, this is obviously true). This is the motivation behind considering a specific Bernoulli bond percolation on \(B(G)\) and comparing the corresponding connection probabilities:\vspace{10pt}\pagebreak\\
Given a set \(T \subseteq V(G)\) of vertices and some \emph{percolation vector} \(\vec{p} = (p_e)_{e\in E(G)} \in [0,1]^{E(G)}\) of edge weights on \(G\), we consider a random subgraph \(B^T \vcentcolon= B^T(G) \vcentcolon= B^T(G,\vec{p})\) of \(B(G)\) in which all vertices of \(B(G)\) are present, the horizontal edges \(e_i\) of \(B(G)\) are randomly present independently of each other with probability \(p_e\) and every post \(v_0v_1\) of \(B(G)\) is deterministically present if and only if \(v \in T\).\vspace{6pt}
\begin{center}
\includegraphics[scale=0.35]{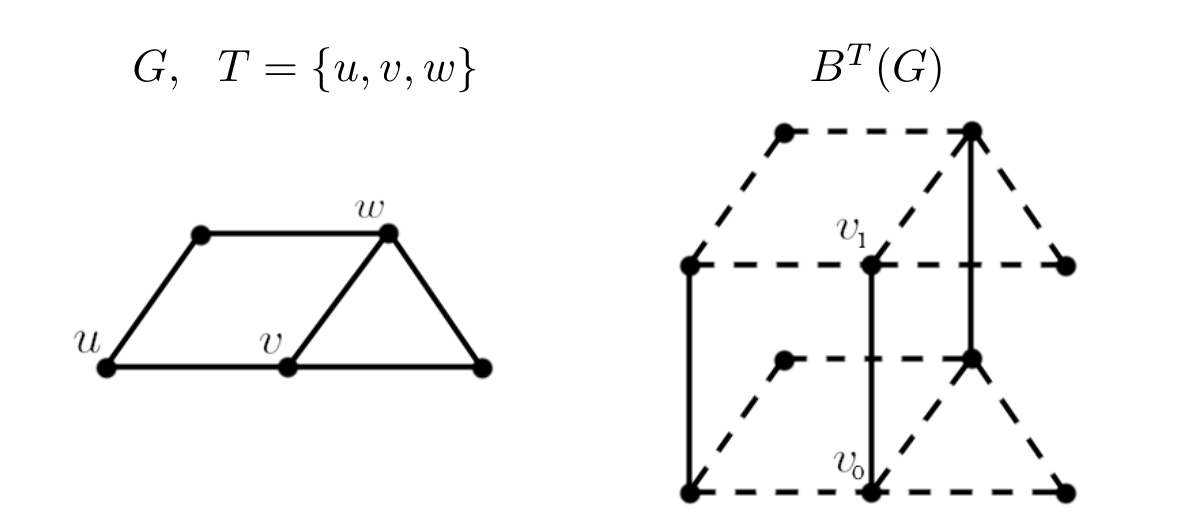}\vspace{9pt}\\
\emph{Fig. 1:\text{ } Bernoulli bond percolation on a bunkbed graph}
\end{center}\vspace{7pt}
Note that for any edge \(e \in E(G)\), both copies \(e_0\) and \(e_1\) are present with the same probability. For any two vertices \(v, w \in V(G)\), we now let \(\mathscr{B}(G, T, \vec{p}, v, w)\) denote the statement that
\begin{center}\(\mathbb{P}\bigl( v_0 \overset{B^T}{\longleftrightarrow} w_0\bigr) \quad \geq \quad \mathbb{P}\bigl( v_0 \overset{B^T}{\longleftrightarrow} w_1\bigr)\)\end{center}holds, i.e. that a connection between \(v_0\) and \(w_0\) in \(B^T\) is at least as likely as a connection between \(v_0\) and \(w_1\). Note that whether or not \(\mathscr{B}(G, T, \vec{p}, v, w)\) holds depends only on \((v, w)\) and the distribution of \(B^T\), but not on its explicit realization, justifying the notation.\vspace{6pt}\\
We say that a given graph \(G\) satisfies the \emph{strong bunkbed conjecture} (as proposed by Pieter Kasteleyn in 1985), if and only if for all \(T \subseteq V(G)\), \(\vec{p} \in [0,1]^{E(G)}\) and \(v, w \in V(G)\), the inequality \(\mathscr{B}(G, T, \vec{p}, v, w)\) holds.\vspace{6pt}\\
We say that \(G\) satisfies the \emph{weak bunkbed conjecture}, if and only if for all \(T \subseteq V(G)\), \(p \in [0,1]\) and \(v, w \in V(G)\), the inequality \(\mathscr{B}(G, T, (p)_{e\in E(G)}, v, w)\) holds; \text{i.e. if} and only if it holds in all the special cases where the percolation vector \(\vec{p}\) is constant and thus every horizontal edge is present with the same probability.\vspace{6pt}\\
We will now formulate our main results. For any vertex \(x\in V(G)\) of a graph \(G\), let \(G\setminus x\) denote the graph that arises if the vertex \(x\) and all edges incident to \(x\) are deleted from \(G\).\vspace{6pt}\\
In graph theory, a graph \(G\) is called \emph{biconnected} if and only if \(\vert V(G)\vert \geq 2\), \(G\) is connected and \(G\setminus x\) is still connected for every \(x\in V(G)\). Note that in particular, we consider the complete graph \(K_2\) consisting of exactly two vertices and an edge between them to be biconnected.\vspace{6pt}\\
A subgraph \(H\) of \(G\) is called a \emph{biconnected component} or \emph{block} of \(G\) if and only if \(H\) is maximally biconnected in \(G\); i.e. if and only if \(H\) is biconnected and for every biconnected subgraph \(H'\) of \(G\) such that \(H\) is a subgraph of \(H'\) we have \(H = H'\). Let \(\mathbb{B}(G)\) denote the set of blocks of \(G\). A block \(H\in\mathbb{B}(G)\) is called \emph{trivial} if and only if \(H \cong K_2\) .\vspace{6pt}\\
Note that every edge \(e=vw\in E(G)\) is contained in exactly one block.\footnote{To see this, consider the union \(H\) of all biconnected subgraphs \(\widetilde{G}\) of \(G\) containing \(e\), which is non-empty since \((\{v,w\}, \{e\}) \cong K_2\) itself is biconnected. It is easy to see that \(H\) is still biconnected: For every \(x\in V(H)\), we can pick \(u \in \{v,w\}\setminus \{x\}\), and since every \(z \in V(H\setminus x)\) is contained in a biconnected subgraph \(\widetilde{G}\) of \(G\) containing \(u\), there exists a path from \(z\) to \(u\) in \(\widetilde{G}\setminus x\) which is also a path in \(H\setminus x\). Thus, \(H\setminus x\) is still connected, so \(H\) is biconnected. By definition it is now clear that \(H\) is the only block containing \(e\).} Therefore, every percolation vector \(\vec{p}=(p_e)_{e\in E(G)}\in [0,1]^{E(G)}\) on \(G\) can be decomposed into percolation vectors \((p_e)_{e\in E(H)} \in [0,1]^{E(H)}\) on the blocks \(H\in \mathbb{B}(G)\) in an unambiguous way.\vspace{6pt}\\
The following result forms the heart of this paper:\vspace{4pt}
\begin{thm}
Let \(G\) be a graph and \(\vec{p} = (p_e)_{e\in E(G)} \in [0,1]^{E(G)}\) be a percolation vector.\\
Then, the statement that
\begin{center}\(\forall T\subseteq V(G)\)\emph{: } \(\forall v,w\in V(G)\)\emph{: } \(\mathscr{B}(G, T, \vec{p}, v, w)\) holds\end{center}
is equivalent to
\begin{center}\(\forall H\in\mathbb{B}(G)\)\emph{: } \(\forall T\subseteq V(H)\)\emph{: } \(\forall v,w\in V(H)\)\emph{: } \(\mathscr{B}(H, T, (p_e)_{e\in E(H)}, v, w)\) holds.\end{center}
In particular, any given graph satisfies the \emph{(}weak/strong\emph{)} bunkbed conjecture if and only if all of its biconnected components satisfy the \emph{(}weak/strong\emph{)} bunkbed conjecture.\vspace{6pt}
\end{thm}
Up until now, most results focused on the weak version of the bunkbed conjecture, proving it for complete graphs and a bit more complicated graphs; see for example \cite{van2019bunkbed} or \cite{richthammer2022bunkbed}. Moreover, it has been proved that for every graph \(G\), there always exists an \(\varepsilon = \varepsilon(G)>0\) such that the weak bunkbed conjecture holds for all constant edge weights \(p \leq \varepsilon\) or \(p \geq 1-\varepsilon\), and this result has later been extended to the strong version; see \cite{hutchcroft2023bunkbed} and \cite{hollom2024new}. Recently, however, even the weak version of the bunkbed conjecture has been shown to be false for some graphs; see \cite{gladkov2024bunkbed}, which builds upon results from \cite{hollom2024bunkbed}. We now look for classes of graphs that still satisfy the strong version of the conjecture in particular.\vspace{4pt}\\
For any graph \(G\), a \emph{cycle} in \(G\) is a subgraph \(C\) of \(G\) that is a cycle graph;\footnote{Note that a path of length \(\geq 3\) whose starting point is also its end point only corresponds to a cycle if it visits no vertex other than the starting point for a second time (i.e. not every circuit is a cycle), and a cycle \(C\) in \(G\) is not necessarily the graph induced onto \(V(C)\) by \(G\) (i.e. in \(G\) there might exist chords of \(C\)).\vspace{2pt}} i.e. \(C \cong C_n\) for some \(n\in\mathbb{N}\), \(n \geq 3\). A graph \(G\) is called a \emph{tree graph} if and only if \(G\) is connected and there exist no cycles in \(G\). A graph \(G\) is called a \emph{cactus graph}\footnote{In graph theory \(-\) as opposed to phytology \(-\) every tree is a cactus, so \textbf{Theorem 2} is a natural common generalization of \textbf{Proposition 6} and \textbf{Proposition 7} from the following section.\vspace{2pt}} if and only if \(G\) is connected and any two cycles \(C \neq \widetilde{C}\) in \(G\) share at most one common vertex:\footnote{Equivalently, \(G\) is a cactus graph if and only if \(G\) is connected and any two cycles \(C \neq \widetilde{C}\) in \(G\) share no common edge: \(E(C) \cap E(\widetilde{C}) = \emptyset\). We discuss two more characterizations of cactus graphs later.}  \(\vert V(C) \cap V(\widetilde{C})\vert \leq 1\).
\begin{thm}
All cactus graphs satisfy the strong bunkbed conjecture.\vspace{6pt}
\end{thm}
We will actually even prove a slightly stronger version of this result; see \textbf{Theorem 11}. For the weak bunkbed conjecture, \textbf{Theorem 1} will extend previous results from other authors:
\begin{thm}
If \(G\) is a graph such that every block \(H\in\mathbb{B}(G)\) is either a cycle, complete, complete bipartite, symmetric complete \(k\)-partite or an edge difference of a complete graph and a complete subgraph, then \(G\) satisfies the weak bunkbed conjecture.\vspace{6pt}
\end{thm}
We will also examine counterexamples to the strong bunkbed conjecture, and while proving the following first result we will demonstrate how it might be strengthened in the future.
\begin{thm}
Every counterexample to the strong bunkbed conjecture must contain the diamond graph and, more generally, at least one of the following two graphs as a minor:
\begin{center}\includegraphics[scale=0.114]{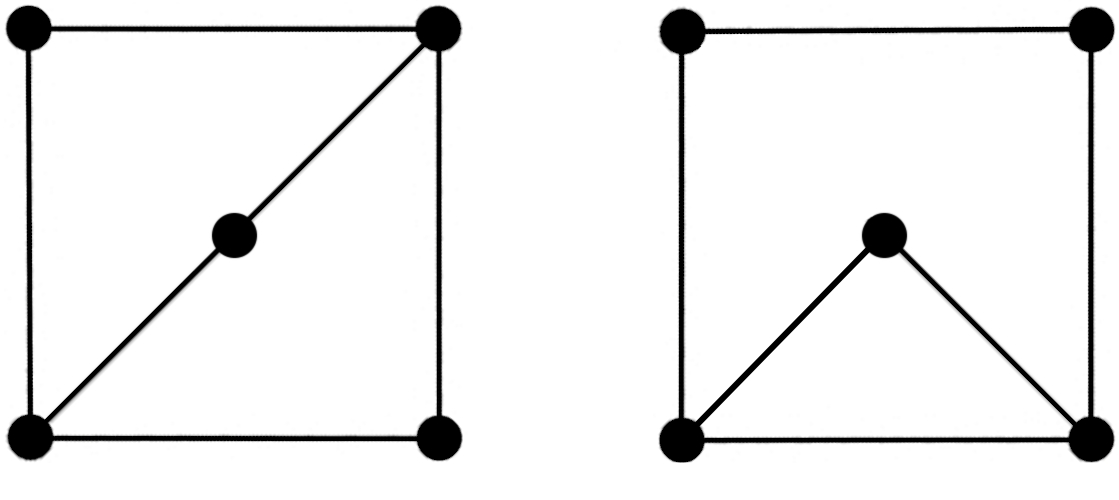}
\end{center}
\end{thm}\pagebreak
For the record, we note that the bunkbed conjecture can be formulated for countably infinite graphs as well.\footnote{If \(G^\ast\) is a random graph consisting of a deterministic and countable set of vertices and a random set of edges between them, one can easily check that the event that fixed vertices \(v\) and \(w\) are connected by a path in \(G^\ast\) is indeed measurable, because it can be written as a countable union of events of the form that some fixed path \(\gamma\) from \(v\) to \(w\) exists in \(G^\ast\), and any such event is a cylinder set in the product-\(\sigma\)-algebra.\vspace{2pt}} The results proved in this paper hold more generally in the case of countable graphs. The proof of the auxiliary \textbf{Theorem 1b} from the next section can be copied, but both directions of the equivalence in \textbf{Theorem 1} have to be adressed seperately and the proof has to be adjusted in a less trivial way than it might seem at first glance.
\section{Proofs}
In this section we give proofs for the stated results. \textbf{Theorem 1} essentially deals with graphs consisting of several small graphs that are glued together at specific vertices one at a time. To start off, we deal with the case of just two such graphs \(G\) and \(H\) being glued together by identifying any one vertex of \(G\) with any one vertex of \(H\). Since we only care about graphs up to graph isomorphism rather than the explicit names of the vertices, we may assume w.l.o.g. that \(\vert V(G) \cap V(H)\vert = 1\). For the shared vertex \(x \in V(G) \cap V(H)\), we define \(G\cup H\) to be the graph that results when \(G\) and \(H\) are glued together at \(x\), i.e. we simply let \(V(G\cup H)\vcentcolon=V(G) \cup V(H)\) and \(E(G\cup H) \vcentcolon= E(G) \cup E(H)\).
\begin{center}\vspace{3pt}
\includegraphics[scale=0.35]{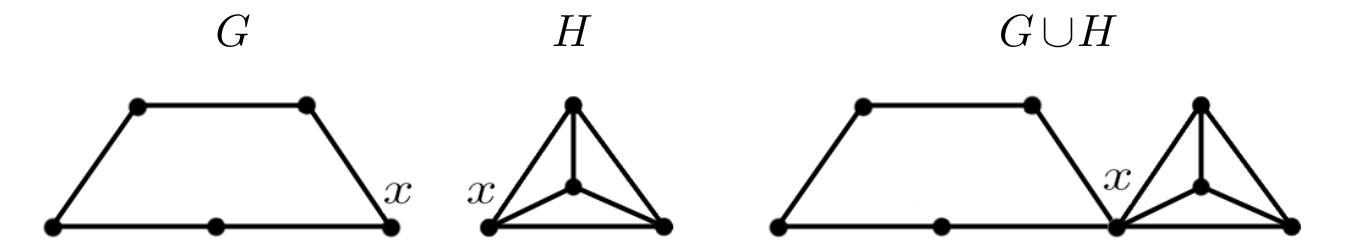}
\vspace{8pt}\\
\emph{Fig. 2: \text{ }\(G \cong C_5\) and \(H \cong K_4\) glued together at some shared vertex}
\end{center}\vspace{12pt}
Note that if \(\vec{p} \in [0,1]^{E(G)}\) and \(\vec{q} \in [0,1]^{E(H)}\) are percolation vectors on \(G\) and \(H\), respectively, we can assemble all of their edge weights in a new and unambiguous percolation vector \((\vec{p},\vec{q}) \in [0,1]^{E(G\text{ }\cup\text{ } H)}\) on \(G\cup H\), because \(E(G) \cap E(H) = \emptyset\). Therefore, our first step will be to prove\footnote{It has been brought to the attention of the author that recently, another work with overlapping results was submitted, see \cite{meunier2024bunkbed}. The overlap mainly consists of special cases of \textbf{Theorem 1b} from above (and immediate corollaries thereof). However, the aforementioned theorem was proven independently from the cited source as part of a bachelor's thesis that this paper is based on, submitted in 2023 at the university of Paderborn under the title "Beweis der Bunkbed-Vermutung für Graphen mit kantendisjunkten Kreisen".} the following version of \textbf{Theorem 1}:
\begin{1b} Let \(G\) and \(H\) be two graphs such that \(\vert V(G) \cap V(H)\vert = 1\). Consider two fixed percolation vectors \(\vec{p} \in [0,1]^{E(G)}\) and \(\vec{q} \in [0,1]^{E(H)}\) on \(G\) and \(H\), respectively.\\
Then, the statement that
\begin{center}\(\forall T \subseteq V(G\cup H)\)\emph{: } \(\forall v, w \in V(G\cup H)\)\emph{: } \(\mathscr{B}(G\cup H, T, (\vec{p},\vec{q}), v, w)\) holds\end{center}
is true if and only if both
\begin{center}\(\forall T \subseteq V(G)\)\emph{: } \(\forall v, w \in V(G)\)\emph{: } \(\mathscr{B}(G, T, \vec{p}, v, w)\) holds\end{center}
and
\begin{center}\(\forall T \subseteq V(H)\)\emph{: } \(\forall v, w \in V(H)\)\emph{: } \(\mathscr{B}(H, T, \vec{q}, v, w)\) holds\end{center}
are true. In particular, \(G\cup H\) satisfies the \emph{(}weak/strong\emph{)} bunkbed conjecture if and only if both \(G\) and \(H\) satisfy the \emph{(}weak/strong\emph{)} bunkbed conjecture.\vspace{6pt}
\end{1b}\vspace{8pt}\pagebreak
The proof of \textbf{Theorem 1b} relies on decomposition arguments that use the symmetrical structure inherent to bunkbed graphs. Note that when considering a graph \(G\) and a random subgraph \(B^T\) of \(B(G)\) as described before, the mirrored random subgraph \(\widetilde{B^T}\) of \(B(G)\) that emerges when we let every vertex of \(B(G)\) be present, let every post \(v_0v_1 \in E(B(G))\) be present deterministically if and only if \(v\in T\) and let every horizontal edge \(e_i \in E(B(G))\) be present if and only if \(e_{1-i}\) is present in \(B^T\) is identically distributed, and for all \(v,w \in V(G)\) and \(i,j \in \{0,1\}\), \(v_i\) and \(w_j\) are connected by a path in \(B^T\) if and only if \(v_{1-i}\) and \(w_{1-j}\) are connected by a path in \(\widetilde{B^T}\). In particular, it follows that
\begin{center}\(\mathbb{P}\bigl( v_i \overset{B^T}{\longleftrightarrow} w_j\bigr)\quad=\quad\mathbb{P}\bigl( v_{1-i} \overset{B^T}{\longleftrightarrow} w_{1-j}\bigr).\)\end{center}
We will refer to arguments similar to this as \emph{mirroring arguments}. Note that even if we only mirror the copies of some of the edges of \(G\) instead of all of them, the distribution of \(B^T\) and \(\widetilde{B^T}\) will always be the same. After the proof, \textbf{Theorem 1} will follow in an elementary and purely graph theoretical way from \textbf{Theorem 1b}, i.e. without actually getting into probability theory again. The crux behind \textbf{Theorem 2} is that any block of a cactus graph is either a cycle or trivial, so \textbf{Theorem 1} reduces the task of proving \textbf{Theorem 2} to the task of proving the strong bunkbed conjecture for all cycle graphs \(C_n\), \(n\geq 3\).\vspace{6pt}
\begin{proof}(Of \textbf{Theorem 1b}.)\vspace{1pt}\\
For every \(T \subseteq V(G\cup H)\), let \(B^T(G\cup H) \vcentcolon= B^T(G\cup H, (\vec{p},\vec{q}))\) be a random graph as described before. W.l.o.g., assume that the horizontal edges of \(B^T(G\cup H)\) are always the same for every such \(T\), i.e. we generate the random horizontal edges only once and then, for every \(T\), we add all the posts \(y_0y_1\) with \(y \in T\) to create \(B^T(G\cup H)\). Moreover, we abbreviate \(T_G \vcentcolon= T \cap V(G)\) and \(T_H \vcentcolon= T \cap V(H)\), and we let \(B^{T_G}(G) \vcentcolon= B^{T_G}(G, \vec{p})\) and \(B^{T_H}(H) \vcentcolon= B^{T_H}(H, \vec{q})\) denote the random graphs induced by \(B^T(G\cup H)\) onto \(V(B(G))\) and \(V(B(H))\), respectively, i.e. any edge of \(B(G)\) is present in \(B^{T_G}(G)\) if and only if it is present in \(B^T(G\cup H)\), and similarly for \(B^{T_H}(H)\). Let \(x \in V(G) \cap V(H)\) be the shared vertex.\vspace{4pt}
\begin{center}
\includegraphics[scale=0.35]{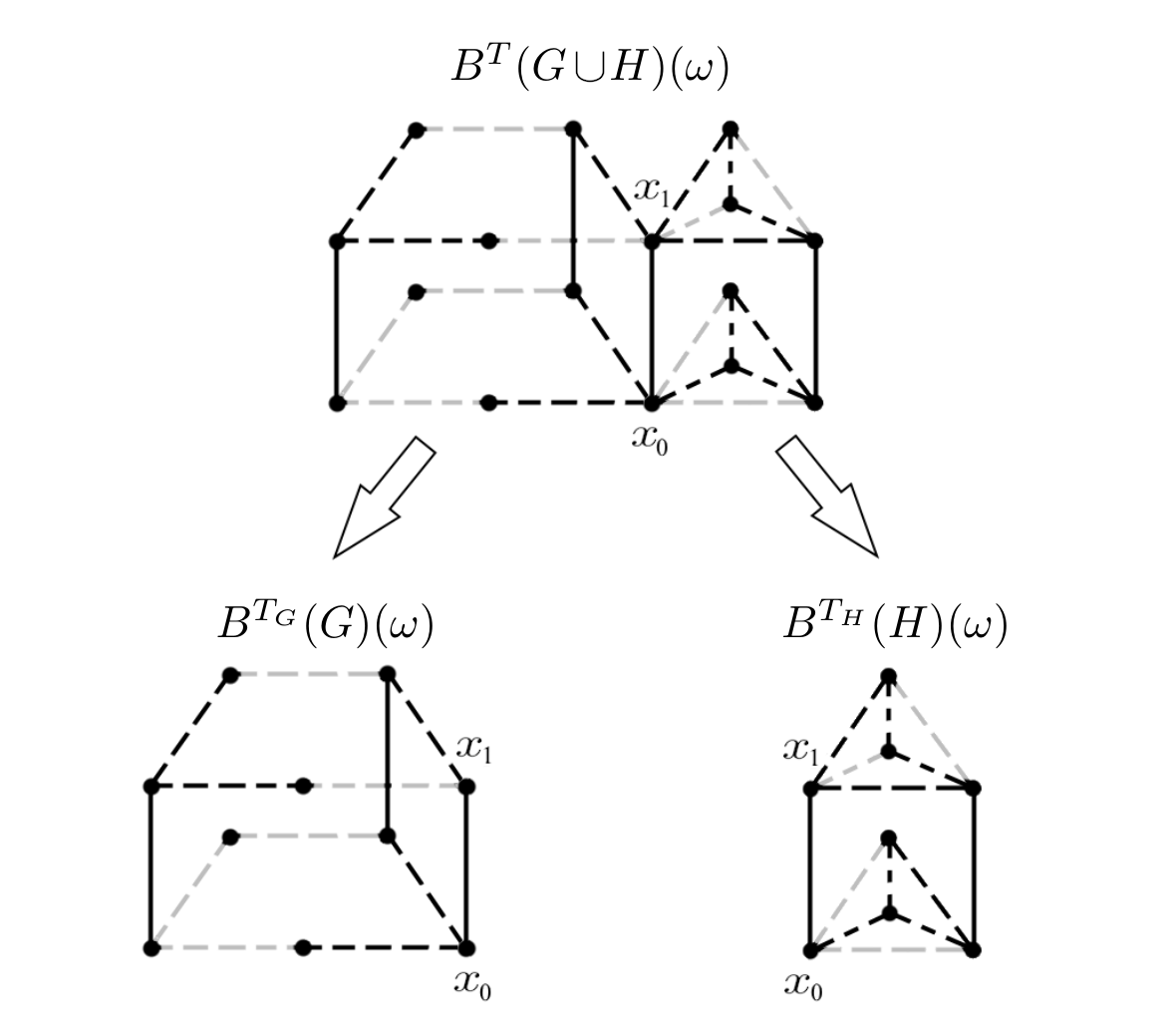}\\
\emph{Fig. 3: \text{ } Induced Bernoulli bond percolation}
\end{center}\vspace{8pt}
Note that for any \(T_1 \subseteq V(G)\) and \(T_2 \subseteq V(H)\), the random graphs \(B^{T_1}(G)\) and \(B^{T_2}(H)\) are now well defined individually, since for every such \(T_1\), there is a \(T \subseteq V(G\cup H)\) such that \(T_G = T_1\), and by design, the edges present in \(B^{T_1}(G)\) don't depend on the choice of such \(T\) (and similarly for \(T_2\)). This means that \(T_1\) and \(T_2\) don't even have to be cut down from the same \(T\); in particular, the post \(x_0x_1\) may be included in \(B^{T_1}(G)\), but not in \(B^{T_2}(H)\). Furthermore, since \(B(G)\) and \(B(H)\) don't share a common horizontal edge, any two such random graphs \(B^{T_1}(G)\) and \(B^{T_2}(H)\) are always independent of each other. We show both directions of the stated equivalence seperately.\vspace{6pt}\\
"\(\Rightarrow\)\text{": } This is the easy direction. If we want to show that the inequality \(\mathscr{B}(G, T_1, \vec{p}, v, w)\) holds for any \(T_1 \subseteq V(G)\) and \(v, w\in V(G)\), then in particular, we have \(T_1\subseteq V(G\cup H)\) and \(v,w \in V(G\cup H)\). It is easy to see that any self-avoiding path \(\gamma\) from \(v_0\) to any \(w_i\) in \(B^{T_1}(G\cup H)\) must never leave \(B^{T_1}(G)\): Otherwise, as outside of \(B^{T_1}(G)\) the path \(\gamma\) can't switch between the bunks \(\bigl(\text{since }T_1\subseteq V(G)\bigr)\), it would visit one copy of \(x\) twice. It follows that for \(i\in\{0,1\}\), we have \begin{center}\(\{\overset{B^{T_1}(G\text{ }\cup\text{ } H)}{v_0 \text{ }\text{ }\longleftrightarrow\text{ }\text{ }w_i}\}\quad=\quad\{v_0 \overset{B^{T_1}(G)}{\longleftrightarrow} w_i\}\).\end{center} Therefore, \(\mathscr{B}(G, T_1, \vec{p}, v, w)\) is equivalent to the assumed inequality \(\mathscr{B}(G\cup H, T_1, (\vec{p},\vec{q}), v, w)\). A similar argument also shows that  \(\mathscr{B}(H, T_2, \vec{q}, v, w)\) holds for any \(T_2 \subseteq V(H)\) and any \(v,w \in V(H)\).\vspace{12pt}\\
"\(\Leftarrow\)\text{": } Let \(v, w \in V(G\cup H)\) and \(T \subseteq V(G\cup H)\). In order to show that the inequality \(\mathscr{B}(G\cup H, T, (\vec{p},\vec{q}), v, w)\) holds, we distinguish between the following two cases:\vspace{12pt}\\
\emph{\underline{Case 1}: Both of the vertices \(v\) and \(w\) lie in \(V(G)\) or both lie in \(V(H)\).}\vspace{10pt}\\
W.l.o.g., assume that \(v, w \in V(G)\). First, we note that for \(i \in \{0,1\}\), every self-avoiding path from \(v_0\) to \(w_i\) that passes through one of the copies of \(x\) and visits more vertices from \(V(H) \times \{0,1\}\) has to eventually come back and pass through the other copy of \(x\). As such, the only way that the random graph \(B^{T_H}(H)\) can contribute to a connection between \(v_0\) and \(w_i\) in \(B^T(G\cup H)\) is if it allows a path to switch between the bunks of \(B^{T_G}(G)\) by providing a connection between \(x_0\) and \(x_1\). In particular, on the event \(\{x_0 \centernot{\overset{B^{T_H}(H)}{\longleftrightarrow}} x_1\}\), \(v_0\) and \(w_i\) being connected by a path in \(B^T(G\cup H)\) is equivalent to \(v_0\) and \(w_i\) being connected by a path in \(B^{T_G}(G)\), while on the event \(\{x_0 \overset{B^{T_H}(H)}{\longleftrightarrow} x_1\}\), \(v_0\) and \(w_i\) being connected by a path in \(B^T(G\cup H)\) is equivalent to \(v_0\) and \(w_i\) being connected by a path in \(B^{T_G\text{ }\cup\text{ }\{x\}}(G)\).\vspace{6pt}\\Furthermore, by the assumptions of this direction of the theorem, we know that both \(\mathscr{B}(G, T_G, \vec{p}, v, w)\) and \(\mathscr{B}(G, T_G\cup\{x\}, \vec{p}, v, w)\) hold.\vspace{4pt}\begin{spacing}{1.3}\noindent Together, all of this shows that
\begin{align*}
&\mathbb{P}\bigl(\overset{B^T(G\text{ }\cup\text{ } H)}{v_0 \text{ }\text{ }\longleftrightarrow\text{ }\text{ }w_0}\bigr)\\
= \quad &\mathbb{P}\bigl(\overset{B^T(G\text{ }\cup\text{ } H)}{v_0 \text{ }\text{ }\longleftrightarrow\text{ }\text{ }w_0}, \text{ } x_0 \centernot{\overset{B^{T_H}(H)}{\longleftrightarrow}} x_1\bigr) \quad\text{ } &&&+&\quad\quad \mathbb{P}\bigl( \overset{B^T(G\text{ }\cup\text{ } H)}{v_0 \text{ }\text{ }\longleftrightarrow\text{ }\text{ }w_0}, \text{ } x_0 \overset{B^{T_H}(H)}{\longleftrightarrow} x_1\bigr)\\
=\quad &\mathbb{P}\bigl( v_0 \overset{B^{T_G}(G)}{\longleftrightarrow} w_0, \text{ } x_0 \centernot{\overset{B^{T_H}(H)}{\longleftrightarrow}} x_1\bigr)\quad\quad &&&+&\quad\quad \mathbb{P}\bigl( \overset{B^{T_G \text{ }\cup \text{ }\{x\}}(G)}{v_0 \quad\longleftrightarrow\quad w_0}, \text{ }  x_0 \overset{B^{T_H}(H)}{\longleftrightarrow} x_1\bigr)\\
=\quad &\mathbb{P}\bigl( v_0 \overset{B^{T_G}(G)}{\longleftrightarrow} w_0\bigr)\text{ } \cdot \text{ }\mathbb{P}\bigl(x_0 \centernot{\overset{B^{T_H}(H)}{\longleftrightarrow}} x_1\bigr) &&&+&\quad\quad \mathbb{P}\bigl( \overset{B^{T_G \text{ }\cup \text{ }\{x\}}(G)}{v_0 \quad\longleftrightarrow\quad w_0}\bigr)\text{ } \cdot \text{ }\mathbb{P}\bigl(x_0 \overset{B^{T_H}(H)}{\longleftrightarrow} x_1\bigr)\\
\geq\quad &\mathbb{P}\bigl( v_0 \overset{B^{T_G}(G)}{\longleftrightarrow} w_1\bigr)\text{ } \cdot \text{ }\mathbb{P}\bigl(x_0 \centernot{\overset{B^{T_H}(H)}{\longleftrightarrow}} x_1\bigr) &&&+&\quad\quad \mathbb{P}\bigl( \overset{B^{T_G \text{ }\cup \text{ }\{x\}}(G)}{v_0 \quad\longleftrightarrow\quad w_1}\bigr)\text{ } \cdot \text{ }\mathbb{P}\bigl(x_0 \overset{B^{T_H}(H)}{\longleftrightarrow} x_1\bigr)\\
=\quad &\mathbb{P}\bigl( v_0 \overset{B^{T_G}(G)}{\longleftrightarrow} w_1, \text{ } x_0 \centernot{\overset{B^{T_H}(H)}{\longleftrightarrow}} x_1\bigr) \quad\quad &&&+&\quad\quad \mathbb{P}\bigl( \overset{B^{T_G \text{ }\cup \text{ }\{x\}}(G)}{v_0 \quad\longleftrightarrow\quad w_1}, \text{ }  x_0 \overset{B^{T_H}(H)}{\longleftrightarrow} x_1\bigr)\\
= \quad &\mathbb{P}\bigl( \overset{B^T(G\text{ }\cup\text{ } H)}{v_0 \text{ }\text{ }\longleftrightarrow\text{ }\text{ }w_1}, \text{ } x_0 \centernot{\overset{B^{T_H}(H)}{\longleftrightarrow}} x_1\bigr) \quad\text{ } &&&+&\quad\quad \mathbb{P}\bigl( \overset{B^T(G\text{ }\cup\text{ } H)}{v_0 \text{ }\text{ }\longleftrightarrow\text{ }\text{ }w_1}, \text{ } x_0 \overset{B^{T_H}(H)}{\longleftrightarrow} x_1\bigr)\\
=\quad &\overset{\text{ }}{\mathbb{P}\bigl( \overset{B^T(G\text{ }\cup\text{ } H)}{v_0 \text{ }\text{ }\longleftrightarrow\text{ }\text{ }w_1}\bigr),}
\end{align*}
\noindent concluding the first case.\end{spacing}\pagebreak
\noindent\emph{\underline{Case 2}: One of the vertices \(v\) and \(w\) lies in \(V(G)\setminus\{x\}\) and the other one lies in \(V(H)\setminus\{x\}\).}\vspace{10pt}\\
W.l.o.g., assume \(v\in V(G)\setminus\{x\}\), \(w\in V(H)\setminus\{x\}\). Since  \(\mathscr{B}(G, T_G, \vec{p}, v, x)\) holds, we have
\begin{align*}
&\mathbb{P}\bigl(v_0 \overset{B^{T_G}(G)}{\longleftrightarrow} x_0\bigr) \text{ }\text{ } = \text{ }\text{ } \mathbb{P}\bigl(v_0 \overset{B^{T_G}(G)}{\longleftrightarrow} x_0\text{, }v_0 \centernot{\overset{B^{T_G}(G)}{\longleftrightarrow}} x_1\bigr) \text{ }\text{ } + \text{ }\text{ } \mathbb{P}\bigl(v_0 \overset{B^{T_G}(G)}{\longleftrightarrow} x_0\text{, }v_0 \overset{B^{T_G}(G)}{\longleftrightarrow} x_1\bigr)\\
\geq \text{ }\text{ } &\mathbb{P}\bigl(v_0 \overset{B^{T_G}(G)}{\longleftrightarrow} x_1\bigr) \text{ }\text{ } = \text{ }\text{ } \mathbb{P}\bigl(v_0 \centernot{\overset{B^{T_G}(G)}{\longleftrightarrow}} x_0\text{, }v_0 \overset{B^{T_G}(G)}{\longleftrightarrow} x_1\bigr) \text{ }\text{ } + \text{ }\text{ } \mathbb{P}\bigl(v_0 \overset{B^{T_G}(G)}{\longleftrightarrow} x_0\text{, }v_0 \overset{B^{T_G}(G)}{\longleftrightarrow} x_1\bigr),
\end{align*}
and as such,
\begin{center}
\(Q \quad  \vcentcolon= \quad \mathbb{P}\bigl(v_0 \overset{B^{T_G}(G)}{\longleftrightarrow} x_0\text{, }v_0 \centernot{\overset{B^{T_G}(G)}{\longleftrightarrow}} x_1\bigr) \quad \geq \quad \mathbb{P}\bigl(v_0 \centernot{\overset{B^{T_G}(G)}{\longleftrightarrow}} x_0\text{, }v_0 \overset{B^{T_G}(G)}{\longleftrightarrow} x_1\bigr) \quad =\vcentcolon \quad q\).
\end{center}\vspace{3pt}
Moreover, a mirroring argument implies the following identities:
\begin{center}
\(R \quad \vcentcolon= \quad \mathbb{P}\bigl(x_0 \overset{B^{T_H}(H)}{\longleftrightarrow} w_0\bigr) \quad = \quad \mathbb{P}\bigl(x_1 \overset{B^{T_H}(H)}{\longleftrightarrow} w_1\bigr)\)\text{, }
\end{center}
\begin{center}
\(r \quad \vcentcolon= \quad \mathbb{P}\bigl(x_0 \overset{B^{T_H}(H)}{\longleftrightarrow} w_1\bigr) \quad = \quad \mathbb{P}\bigl(x_1 \overset{B^{T_H}(H)}{\longleftrightarrow} w_0\bigr)\),
\end{center}
\begin{center}
\(c \quad \vcentcolon= \quad\mathbb{P}\bigl(\{x_0\text{, }x_1\} \overset{B^{T_H}(H)}{\longleftrightarrow} w_0\bigr) \quad= \quad \mathbb{P}\bigl(\{x_0\text{, }x_1\} \overset{B^{T_H}(H)}{\longleftrightarrow} w_1\bigr)\).
\end{center}\vspace{2pt}
Since we also assumed that \(\mathscr{B}(H, T_H, \vec{q}, x, w)\) holds, it follows that \(R\geq r\).\vspace{4pt}\\
Note that every self-avoiding path from \(v _0\) to \(w_0\) in \(B(G\cup T)\) must visit atleast one copy of \(x\), and once it reached a vertex from \(\bigl(V(H)\setminus\{x\}\bigr) \times \{0,1\}\), it can never visit any vertices from \(\bigl(V(G)\setminus\{x\}\bigr) \times \{0,1\}\) again, because otherwise it would eventually have to visit one of the copies of \(x\) a second time. As such, any self-avoiding path from \(v_0\) to \(w_0\) in \(B(G\cup H)\) can be decomposed into a first section connecting \(v_0\) with some copy of \(x\) while never leaving \(B(G)\) and a second section connecting this copy of \(x\) with \(w_0\) while never leaving \(B(H)\). Since \(B^{T_G}(G)\) and \(B^{T_H}(H)\) are independent, this shows the following decomposition:
\begin{align*}
\mathbb{P}\bigl( \overset{B^T(G\text{ }\cup\text{ } H)}{v_0 \text{ }\text{ }\longleftrightarrow\text{ }\text{ }w_0}\bigr) \quad =& \quad \mathbb{P}\bigl(v_0 \overset{B^{T_G}(G)}{\longleftrightarrow} x_0\text{, }v_0 \overset{B^{T_G}(G)}{\longleftrightarrow} x_1\bigr) \text{ }\text{ } \cdot \text{ }\text{ } \mathbb{P}\bigl(\{x_0\text{, }x_1\} \overset{B^{T_H}(H)}{\longleftrightarrow} w_0\bigr)\\
+& \quad \mathbb{P}\bigl(v_0 \overset{B^{T_G}(G)}{\longleftrightarrow} x_0\text{, }v_0 \centernot{\overset{B^{T_G}(G)}{\longleftrightarrow}} x_1\bigr) \text{ }\text{ } \cdot \text{ }\text{ } \mathbb{P}\bigl(x_0 \overset{B^{T_H}(H)}{\longleftrightarrow} w_0\bigr)\\
+& \quad \mathbb{P}\bigl(v_0 \centernot{\overset{B^{T_G}(G)}{\longleftrightarrow}} x_0\text{, }v_0 \overset{B^{T_G}(G)}{\longleftrightarrow} x_1\bigr) \text{ }\text{ } \cdot \text{ }\text{ } \mathbb{P}\bigl(x_1 \overset{B^{T_H}(H)}{\longleftrightarrow} w_0\bigr),
\end{align*}
or according to the definitions above,\vspace{2pt}
\begin{center}
\(\mathbb{P}\bigl( \overset{B^T(G\text{ }\cup\text{ } H)}{v_0 \text{ }\text{ }\longleftrightarrow\text{ }\text{ }w_0}\bigr) \quad = \quad \mathbb{P}\bigl(v_0 \overset{B^{T_G}(G)}{\longleftrightarrow} x_0\text{, }v_0 \overset{B^{T_G}(G)}{\longleftrightarrow} x_1\bigr) \cdot c \quad + \quad Q \cdot R \quad + \quad q \cdot r\).
\end{center}\vspace{3pt}
Using similar reasoning, it also follows that
\begin{align*}
\mathbb{P}\bigl( \overset{B^T(G\text{ }\cup\text{ } H)}{v_0 \text{ }\text{ }\longleftrightarrow\text{ }\text{ }w_1}\bigr) \quad =& \quad \mathbb{P}\bigl(v_0 \overset{B^{T_G}(G)}{\longleftrightarrow} x_0\text{, }v_0 \overset{B^{T_G}(G)}{\longleftrightarrow} x_1\bigr) \text{ }\text{ } \cdot \text{ }\text{ } \mathbb{P}\bigl(\{x_0\text{, }x_1\} \overset{B^{T_H}(H)}{\longleftrightarrow} w_1\bigr)\\
+& \quad \mathbb{P}\bigl(v_0 \overset{B^{T_G}(G)}{\longleftrightarrow} x_0\text{, }v_0 \centernot{\overset{B^{T_G}(G)}{\longleftrightarrow}} x_1\bigr) \text{ }\text{ } \cdot \text{ }\text{ } \mathbb{P}\bigl(x_0 \overset{B^{T_H}(H)}{\longleftrightarrow} w_1\bigr)\\
+& \quad \mathbb{P}\bigl(v_0 \centernot{\overset{B^{T_G}(G)}{\longleftrightarrow}} x_0\text{, }v_0 \overset{B^{T_G}(G)}{\longleftrightarrow} x_1\bigr) \text{ }\text{ } \cdot \text{ }\text{ } \mathbb{P}\bigl(x_1 \overset{B^{T_H}(H)}{\longleftrightarrow} w_1\bigr),
\end{align*}
which we restate as\vspace{2pt}
\begin{center}
\(\mathbb{P}\bigl( \overset{B^T(G\text{ }\cup\text{ } H)}{v_0 \text{ }\text{ }\longleftrightarrow\text{ }\text{ }w_1}\bigr) \quad = \quad \mathbb{P}\bigl(v_0 \overset{B^{T_G}(G)}{\longleftrightarrow} x_0\text{, }v_0 \overset{B^{T_G}(G)}{\longleftrightarrow} x_1\bigr) \cdot c \quad + \quad Q \cdot r \quad + \quad q \cdot R\).
\end{center}\vspace{3pt}
In particular, it follows that\vspace{2pt}
\begin{center}
\(\mathbb{P}\bigl( \overset{B^T(G\text{ }\cup\text{ } H)}{v_0 \text{ }\text{ }\longleftrightarrow\text{ }\text{ }w_0}\bigr) \text{ }-\text{ }\mathbb{P}\bigl( \overset{B^T(G\text{ }\cup\text{ } H)}{v_0 \text{ }\text{ }\longleftrightarrow\text{ }\text{ }w_1}\bigr) \text{ } = \text{ } QR \text{ }+\text{ } qr \text{ }-\text{ } Qr \text{ }-\text{ } qR \text{ }= \text{ }(Q - q)(R - r)\text{ }\geq \text{ } 0\),
\end{center}\vspace{3pt}
as desired.
\end{proof}
\noindent Now we show how to extract the original theorem from the special case above.
\begin{proof} (Of \textbf{Theorem 1}.)\vspace{1pt}\\
Assume w.l.o.g. that \(G\) is connected. We show that the stated equivalence holds via strong induction over \(n \vcentcolon= \vert V(G)\vert\in \mathbb{N}_0\). The cases \(n = 0\) and \(n = 1\) are both trivial.\vspace{6pt}\\
Let \(n \geq 2\) and suppose that the equivalence is proved for all \(\widetilde{n} \leq n-1\). We want to show that the equivalence holds for \(n\) as well. As \(G\) is connected and \(n \geq 2\), we have \(\vert\mathbb{B}(G)\vert \geq 1\).\vspace{6pt}\\
If \(\vert \mathbb{B}(G)\vert = 1\), then \(G\) is itself its only block and the desired equivalence is a trivial tautology.\vspace{6pt}\\
If \(\vert\mathbb{B}(G)\vert \geq 2\), then in particular, \(G\) is not biconnected. Therefore, there exists an \(x\in V(G)\) such that \(G\setminus x\) has atleast two different connected components \(U_1\text{, ... , }U_k\) for some \(k \geq 2\). Let \(G_1\)\text{, ... , }\(G_k\) denote the graphs induced by \(G\) onto \(V(U_1)\) \(\cup\) \(\{x\}\)\text{, ... , }\(V(U_k)\) \(\cup\) \(\{x\}\), respectively, which only overlap at the vertex \(x\). Since \(k\geq2\), we have \(\vert V(G_i)\vert \leq \vert V(G)\vert - 1\) for every \(i\in\{1\), ... , \(k\}\), and by considering the cutset \(\{x\}\) it is easy to see that any block of any \(G_i\) is also a block of \(G\). On the other hand, the cutset \(\{x\}\) also shows that every block of \(G\) is in turn a block of some \(G_i\) as well, so it follows that \(\mathbb{B}(G)\) \(=\) \(\mathbb{B}(G_1)\) \(\cup\) ... \(\cup\) \(\mathbb{B}(G_k)\).\vspace{6pt}\\
Since \(G\text{ }=\text{ }G_1\) \(\cup\) ... \(\cup\) \(G_k\) and the \(G_i\) pairwise do not overlap at any vertex other than \(x\), \(k-1\) applications of \textbf{Theorem 1b} show that the statement
\begin{center}\(\forall T\subseteq V(G)\):\text{ } \(\forall v,w\in V(G)\):\text{ } \(\text{ } \mathscr{B}(G, T, \vec{p}, v, w)\)\end{center}
is equivalent to the statement
\begin{center}\(\forall i\in\{1\text{, ... , }k\}\): \(\text{ } \forall T\subseteq V(G_i)\): \(\text{ } \forall v,w\in V(G_i)\): \(\text{ } \mathscr{B}(G_i, T, (p_e)_{e\in E(G_i)}, v, w)\).\end{center}
Since all \(G_i\) have at most \(n-1\) vertices, it follows by the induction hypothesis that the statement above in turn is equivalent to
\begin{center}\(\forall  i\in\{1\text{, ... , }k\}\): \(\text{ } \forall H\in \mathbb{B}(G_i)\): \(\text{ } \forall T\subseteq V(H)\): \(\text{ } \forall v,w\in V(H)\): \(\text{ } \mathscr{B}(H, T, (p_e)_{e\in E(H)}, v, w)\).\end{center}
Finally, since \(\mathbb{B}(G)\) \(=\) \(\mathbb{B}(G_1)\) \(\cup\) ... \(\cup\) \(\mathbb{B}(G_k)\), the statement above is equivalent to
\begin{center}\(\forall H\in \mathbb{B}(G)\): \(\text{ } \forall T\subseteq V(H)\): \(\text{ } \forall v,w\in V(H)\): \(\text{ } \mathscr{B}(H, T, (p_e)_{e\in E(H)}, v, w)\),\end{center}
concluding the proof.
\end{proof}\vspace{6pt}
\noindent In the following, we show the strong bunkbed conjecture for (slightly generalized) cactus graphs; in particular, we prove \textbf{Theorem 2}. We start with some technical statements.\vspace{4pt}
\begin{A}
Let \(G\) be a graph, \(T \subseteq V(G)\), \(\vec{p}\in [0,1]^{E(G)}\) and \(v,w \in V(G)\).\\
If \(v \in T\) or \(w \in T\) or \(\vert T \vert \in \{0,1\}\), then \(\mathscr{B}(G, T, \vec{p}, v, w)\) holds.\vspace{6pt}
\end{A}
\begin{proof}
The cases \(w \in T\) and \(\vert T \vert = 0\) are both trivial, and the case \(v \in T\) can be dealt with using a simple mirroring argument. Assume now that \(T = \{x\}\) for some \(x \in V(G)\setminus\{v,w\}\). Consider a random graph \(B^T \vcentcolon= B^T(G, \vec{p})\) as usual. Let \(B^T_0\) and \(B^T_1\) denote the independent random graphs that \(B^T\) induces onto \(V(G) \times \{0\}\) and \(V(G) \times \{1\}\), respectively. We will use Harris' inequality (see \emph{Lemma 4.1} and its corollary from \cite{harris1960lower}, which can be proved similarly for any arbitrary finite graph). Since the distributions of \(B^T_0\) and \(B^T_1\) are essentially the same (apart from living on different bunks), we get that\vspace{8pt}\\
\text{ }\text{ }\quad\quad\(\text{ }\mathbb{P}\bigl( v_0 \overset{B^T}{\longleftrightarrow} w_1\bigr) \text{ }\text{ } \overset{\overset{T=\{x\}}{\text{ }}}{=} \text{ }\text{ } \mathbb{P}\bigl( v_0 \overset{B^T_0}{\longleftrightarrow} x_0\text{, }x_1 \overset{B^T_1}{\longleftrightarrow} w_1\bigr) \text{ }\text{ } = \text{ }\text{ } \mathbb{P}\bigl( v_0 \overset{B^T_0}{\longleftrightarrow} x_0\bigr)\mathbb{P}\bigl( x_1 \overset{B^T_1}{\longleftrightarrow} w_1\bigr)\)\vspace{6pt}\\
\text{ }\(=\text{ }\text{ } \mathbb{P}\bigl( v_0 \overset{B^T_0}{\longleftrightarrow} x_0\bigr)\mathbb{P}\bigl( x_0 \overset{B^T_0}{\longleftrightarrow} w_0\bigr) \text{ } \overset{\text{Harris}}{\leq} \text{ }  \mathbb{P}\bigl( v_0 \overset{B^T_0}{\longleftrightarrow} x_0\text{, } x_0 \overset{B^T_0}{\longleftrightarrow} w_0\bigr) \text{ }\text{ } \leq \text{ }\text{ } \mathbb{P}\bigl( v_0 \overset{B^T}{\longleftrightarrow} w_0\bigr)\).
\end{proof}\vspace{10pt}
The next lemma uses ideas that are also found in \cite{richthammer2022bunkbed}.\pagebreak\\
For a graph \(G\) and an edge \(e\in E(G)\), let \(G\setminus e\) denote the graph resulting from deleting \(e\) from \(G\), i.e. \(V(G\setminus e)\vcentcolon= V(G)\) and \(E(G\setminus e) \vcentcolon = E(G)\setminus \{e\}\).
\begin{B}
Let \(G\) be a graph, \(T \subseteq V(G)\), \(\vec{p} = (p_e)_{e\in E(G)} \in [0,1]^{E(G)}\) and \(v,w \in V(G)\).\\
If  \(e^\ast \vcentcolon= vw\in E(G)\) \emph{(}i.e. \(v\) and \(w\) are neighbours\emph{)} and \(\mathscr{B}(G\setminus e^\ast, T, (p_e)_{e\in E(G)\setminus \{e^\ast\}}, v, w)\) holds,
then \(\mathscr{B}(G, T, \vec{p}, v, w)\) holds as well.\vspace{6pt}
\end{B}
\begin{proof}
Let \(B^T(G) \vcentcolon= B^T(G,\vec{p})\) denote a random graph as usual. Furthermore, we will let \(B^T(G\setminus e^\ast) \vcentcolon = B^T(G\setminus e^\ast, (p_e)_{e\in E(G)\setminus\{e^\ast\}})\) denote the random subgraph of \(B(G\setminus e^\ast)\) that is cut down from \(B^T(G)\) in the sense that an edge from \(B(G\setminus e^\ast)\) is present in \(B^T(G\setminus e^\ast)\) if and only if it is present in \(B^T(G)\). Moreover, we abbreviate the event 
\begin{center}
\(A \quad \vcentcolon=\quad \{\)At least one of the two edges \( e^{\ast}_0\) and \(e^{\ast}_1\) is present in \(B^T(G)\)\(\}\).
\end{center}
We will consider two cases: By the law of total probability, the inequality \(\mathscr{B}(G, T, \vec{p}, v, w)\) will follow if we can show that both
\begin{center}\(\mathbb{P}\bigl( v_0 \overset{B^T(G)}{\longleftrightarrow} w_0\mid A\bigr)\quad \overset{!}{\geq} \quad\mathbb{P}\bigl( v_0 \overset{B^T}{\longleftrightarrow} w_1\mid A\bigr)\)\end{center}
and
\begin{center}\(\mathbb{P}\bigl( v_0 \overset{B^T(G)}{\longleftrightarrow} w_0\mid A^c\bigr)\quad \overset{!}{\geq} \quad\mathbb{P}\bigl( v_0 \overset{B^T}{\longleftrightarrow} w_1\mid A^c\bigr)\)\end{center}
hold. (Note that if \(\mathbb{P}(A) \in \{0,1\}\), then we only need to consider the well-defined case.) By the assumption, \(\mathscr{B}(G\setminus e^\ast, T, (p_e)_{e\in E(G)\setminus \{e^\ast\}}, v, w)\) holds. If \(\mathbb{P}(A) \neq 1\), we instantly get that
\begin{center}
\(\mathbb{P}\bigl( v_0 \overset{B^T(G)}{\longleftrightarrow} w_0 \mid A^c\bigr) \text{ }=\text{ }\mathbb{P}\bigl( v_0 \overset{B^T(G\setminus e^\ast)}{\longleftrightarrow} w_0\bigr) \text{ }\geq\text{ } \mathbb{P}\bigl( v_0 \overset{B^T(G\setminus e^\ast)}{\longleftrightarrow} w_1\bigr)\text{ }= \text{ }\mathbb{P}\bigl( v_0 \overset{B^T(G)}{\longleftrightarrow} w_1 \mid A^c\bigr)\).
\end{center}\vspace{2pt}
In order to conclude the proof, all that is left to show is that if \(\mathbb{P}(A) \neq 0\), then
\begin{center}\(d \text{ }\text{ }\vcentcolon=\text{ }\text{ } \mathbb{P}\bigl( u_0 \overset{B^T(G)}{\longleftrightarrow} v_0 \mid A\bigr) \text{ }- \text{ } \mathbb{P}\bigl( u_0 \overset{B^T(G)}{\longleftrightarrow} v_1 \mid A\bigr) \text{ }\text{ }\overset{!}{\geq} \text{ }\text{ }0\).
\end{center}\vspace{2pt}
 To do that, we abbreviate \(B^T \vcentcolon= B^T(G)\) again and for every \(W \subseteq \{v_0, v_1, w_0, w_1\}\), we define\vspace{10pt}\\
\text{ }\text{ }\text{ }\quad\quad\quad\quad\(Q(W) \text{ }\text{ } \vcentcolon= \text{ }\text{ } \mathbb{P}\Bigl(\substack{\text{ In \(B^T\), all the vertices in \(W\) are connected with each other, }\\\text{ but not with any of the other vertices in \( \{v_0, v_1, w_0, w_1\}\setminus W\)}}\mid A\text{ }\Bigr)\).\vspace{10pt}\\
Using this notation, we can decompose several connection probabilities in the following way:\vspace{6pt}\\
\(\mathbb{P}\bigl( v_0 \overset{B^T}{\longleftrightarrow} w_0\mid A\bigr) \text{ }=\text{ } Q(\{v_0,w_0\}) +  \colorbox{sun}{\(Q(\{v_0,v_1,w_0\})\)} +  \colorbox{lavenderblush}{\(Q(\{v_0,w_0,w_1\})\)} + \colorbox{apricot}{\(Q(\{v_0,v_1,w_0,w_1\})\)}\)\vspace{6pt}\\
\(\mathbb{P}\bigl( v_1 \overset{B^T}{\longleftrightarrow} w_1\mid A\bigr) \text{ }=\text{ } Q(\{v_1,w_1\}) + \colorbox{pastelyellow}{\( Q(\{v_0,v_1,w_1\})\)}  + \colorbox{bubbles}{\(Q(\{v_1,w_0,w_1\})\)} + \colorbox{apricot}{\(Q(\{v_0,v_1,w_0,w_1\})\)}\)\vspace{6pt}\\
\(\mathbb{P}\bigl( v_0 \overset{B^T}{\longleftrightarrow} w_1\mid A\bigr) \text{ }=\text{ } Q(\{v_0,w_1\}) + \colorbox{pastelyellow}{\(Q(\{v_0,v_1,w_1\})\)}+ \colorbox{lavenderblush}{\(Q(\{v_0,w_0,w_1\})\)} +  \colorbox{apricot}{\(Q(\{v_0,v_1,w_0,w_1\})\)}\)\vspace{6pt}\\
\(\mathbb{P}\bigl( v_1 \overset{B^T}{\longleftrightarrow} w_0\mid A\bigr) \text{ }=\text{ } Q(\{v_1,w_0\}) + \colorbox{sun}{\(Q(\{v_0,v_1,w_0\})\)} +  \colorbox{bubbles}{\(Q(\{v_1,w_0,w_1\})\)} + \colorbox{apricot}{\(Q(\{v_0,v_1,w_0,w_1\})\)}\)\vspace{10pt}\\
Under the condition of event \(A\), either \(v_0\) is connected to \(w_0\) or \(v_1\) is connected to \(w_1\) in \(B^T\). In particular, it follows that \( Q(\{v_0,w_1\}) = Q(\{v_1,w_0\}) = 0\). Since \(A\) is invariant under mirroring every horizontal edge \(e_i\) with \(e_{1-i}\), we may carefully cancel out terms to see that
\[\text{ }\text{ }2d\text{ }=\text{ }\mathbb{P}\bigl( v_0 \overset{B^T}{\longleftrightarrow} w_0\mid A\bigr)\text{ }+\text{ }\mathbb{P}\bigl( v_1 \overset{B^T}{\longleftrightarrow} w_1\mid A\bigr)\text{ }-\text{ }\mathbb{P}\bigl( v_0 \overset{B^T}{\longleftrightarrow} w_1\mid A\bigr)\text{ }-\text{ }\mathbb{P}\bigl( v_1 \overset{B^T}{\longleftrightarrow} w_0\mid A\bigr)\text{ }\]
\[\text{ }=\text{ }Q(\{v_0,w_0\})\text{ }\text{ }\quad\quad+\text{ }Q(\{v_1,w_1\})\text{ }\text{ }\text{ }\quad\quad-\text{ }Q(\{v_0,w_1\})\text{ }\text{ }\quad\quad-\text{ }Q(\{v_1,w_0\})\quad\vspace{4pt}\]
\[\text{ }=\text{ }Q(\{v_0,w_0\})\text{ }\text{ }\quad\quad+\text{ }Q(\{v_1,w_1\})\vspace{4pt}\quad\quad\quad\quad\quad\quad\quad\quad\quad\quad\quad\quad\quad\quad\quad\quad\quad\quad\quad\text{ }\text{ }\text{ }\]
\(\quad\text{ }\text{ }\text{ }\geq\text{ }0\).
\end{proof}
Now, we want to deal with the case of \(T\) being a vertex seperator for \(v\) and \(w\), which is essentially the setting of the following lemma. For a given graph \(G\) and a set \(W \subseteq V(G)\), let \(\partial_GW \vcentcolon= \{x\in W \mid\exists y\in V(G)\setminus W\): \(xy \in E(G)\}\) denote the \emph{inner boundary} of \(W\) in \(G\).\vspace{1pt}
\begin{C}
Let \(G\) be a graph, \(T \subseteq V(G)\), \(\vec{p} \in [0,1]^{E(G)}\) and \(v,w \in V(G)\). Moreover, let \(B^T \vcentcolon= B^T(G,\vec{p})\) denote a random graph as usual.\\
If there exists a set \(W \subseteq V(G)\) such that \(v \notin W\), \(w \in W\) and \(\partial_GW \subseteq T\), then \begin{center}\(\mathbb{P}\bigl( v_0 \overset{B^T}{\longleftrightarrow} w_0\bigr) \text{ }\text{ } = \text{ }\text{ } \mathbb{P}\bigl( v_0 \overset{B^T}{\longleftrightarrow} w_1\bigr)\).\end{center}
In particular, \(\mathscr{B}(G, T, \vec{p}, v, w)\) holds in this case.\vspace{6pt}
\end{C}
\begin{proof}
Assume w.l.o.g. that \(w \in W\setminus \partial_GW\). (Otherwise, the assumption of the lemma implies \(w \in T\) and the statement is trivial.) We prove the lemma with a refined mirroring argument: Let \(\widetilde{B^T}\) denote the random subgraph of \(B(G)\) that arises if we let every vertex of \(B(G)\) be present, let every post \(v_0v_1\) of \(B(G)\) be present deterministically if and only if \(v \in T\), let every horizontal edge \(e_i=x_iy_i\) of \(B(G)\) with either \(x \notin W\) or \(y \notin W\) be present if and only if it is present in \(B^T\) and let every horizontal edge \(e_i=x_iy_i\) of \(B(G)\) with \(x,y\in W\) be present if and only if \(e_{1-i}\) is present in \(B^T\). Again, \(B^T\) and \(\widetilde{B^T}\) are identically distributed, so it suffices to show that
\begin{center}\(\{v_0 \overset{B^T}{\longleftrightarrow} w_0\}\quad\overset{\text{!}}{=}\quad\{v_0 \overset{\widetilde{B^T}}{\longleftrightarrow} w_1\}\).\end{center}
We show "\(\subseteq\)\text{",} then "\(\supseteq\)\text{" }follows with the same argument.\\
Let \(\omega \in \{v_0 \overset{B^T}{\longleftrightarrow} w_0\}\) and let \(\gamma\) be a path from \(v_0\) to \(w_0\) in \(B^T(\omega)\). Then, for some \(n \in \mathbb{N}\), we decompose \(\gamma = \gamma_1\text{ } \circ\) ...  \(\circ\text{ } \gamma_n\) into sections \(\gamma_k\), where a new section starts whenever \(\gamma\) reaches the next copy of a vertex \(z \in \partial_GW\) (and formally, the end point of every \(\gamma_k\) is the starting point of \(\gamma_{k+1}\)). Since \(v \notin W\) and \(w \in W\setminus\partial_GW\), \(n \geq 2\). Note that the horizontal edges that any one section \(\gamma_k\) leads across are either all mirrored in \(\widetilde{B^T}(\omega)\) or none of them are mirrored at all, and in the first case, \(\gamma_k\) is a path in \(\widetilde{B^T}(\omega)\) as well, while in the second case, the path \(\widetilde{\gamma_k}\) that is mirrored from \(\gamma_k\) is a path in \(\widetilde{B^T}(\omega)\). Also note that \(\gamma_1\) leads across non-mirrored edges and \(\gamma_n\) leads across mirrored edges.  Let \(\eta_1 \vcentcolon= \gamma_1\), so \(\eta_1\) is a path in \(\widetilde{B^T}(\omega)\) that starts from \(v_0\). For \(k \in \{2\), ... , \(n-1\}\), if \(\gamma_k\) only leads across non-mirrored edges, we let \(\eta_k \vcentcolon= \gamma_k\). If \(\gamma_k\) leads over mirrored edges, let \(z, z'\in \partial_GW\) be the vertices such that \(\gamma_k\) starts from a copy of \(z\) and ends in a copy of \(z'\). Then, we define \(\eta_k\) to be the path that starts from the same vertex as \(\gamma_{k}\), leads across the post \(z_0z_1\), then follows \(\widetilde{\gamma_k}\) and then leads across the post \(z'_0z'_1\). Note that in both cases, \(\eta_k\) is now a path in \(\widetilde{B^T}(\omega)\) that starts and ends at the same vertices as \(\gamma_k\). Finally, we let \(\eta_n\) be the path that starts at the starting point \(u_i\) of \(\gamma_{n}\) (for some \(u \in \partial_GW\)), leads across the post \(u_0u_1\) and then follows the mirrored section \(\widetilde{\gamma_n}\). Note that now, since \(\gamma_n\) ends in \(w_0\), \(\eta_n\) ends in \(w_1\). Therefore, \(\eta \vcentcolon= \eta_1\text{ }\circ\) ... \(\circ\text{ }\eta_n\) is a path from \(v_0\) to \(w_1\) in \(\widetilde{B^T}(\omega)\) and as such, it follows that \(\omega \in \{v_0 \overset{\widetilde{B^T}}{\longleftrightarrow} w_1\}\).
\end{proof}\vspace{6pt}
Coming back to a little less abstract graphs, consider the smallest possible block one might encounter when applying \textbf{Theorem 1} and note the following triviality:
\begin{exa}
The complete graph \(K_2\) satisfies the strong bunkbed conjecture.\vspace{6pt}
\end{exa}
\begin{proof}
Let \(T \subseteq V(K_2)\) and \(v, w\in V(K_2)\). Assume w.l.o.g. that \(v \neq w\). Since \(\vert V(K_2)\vert = 2\), either \(v \in T\) or \(w \in T\) or \(T = \emptyset\). Thus, the desired result follows from \textbf{Lemma A}. 
\end{proof}
\begin{prop}
All tree graphs satisfy the strong bunkbed conjecture.\vspace{6pt}
\end{prop}
\begin{proof}
Let \(G\) be a tree and \(H \in \mathbb{B}(G)\). Note that \(H\) is itself a (biconnected) tree. If there existed an \(x\in V(H)\) with two neighbors \(y \neq z\) in \(H\), then the connected graph \(H\setminus x\) would still contain a self-avoiding path from \(y\) to \(z\). The two edges \(xy\) and \(xz\) would close this path to a cycle in \(H\), contradicting that \(H\) is a tree. It follows that \(H\cong K_2\), so by \textbf{Example 5}, \(H\) satisfies the strong bunkbed conjecture, and by \textbf{Theorem 1}, \(G\) satisfies it as well.
\end{proof}
\begin{prop}
Every cycle graph \(C_n\) satisfies the strong bunkbed conjecture. \emph{(}\(n \geq 3\)\emph{)}\vspace{6pt}
\end{prop}
\begin{proof}
Consider a cycle graph \(C\), \(T \subseteq V(C)\), \(\vec{p} = (p_e)_{e\in E(C)} \in [0,1]^{E(C)}\) and \(v,w \in V(C)\), where we w.l.o.g. assume \(v \neq w\).  Consider a random graph \(B^T(C) \vcentcolon= B^T(C,\vec{p})\) as usual.\vspace{6pt}\\
If \(e^\ast \vcentcolon= vw \in E(C)\), then the graph \(C\setminus e^\ast\) is a (linear) tree graph. As such, we first apply \textbf{Proposition 6} to see that \(\mathscr{B}(C\setminus e^\ast, T, (p_e)_{e\in E(C)\setminus \{e^\ast\}}, v, w)\) holds, and then we apply \textbf{Lemma B} to see that \(\mathscr{B}(C, T, \vec{p}, v, w)\) holds as well.\vspace{6pt}\\
If \(vw \notin E(C)\), for some \(N, M \in \mathbb{N}\), we let \(v\), \(x^{(1)}\), ... , \(x^{(N)}\), \(w\) and \(v\), \(y^{(1)}\), ... , \(y^{(M)}\), \(w\) denote the vertices of the only two self-avoiding paths from \(v\) to \(w\) in \(C\). In case there exist \(n \in \{1\), ... , \(N\}\) and \(m\in\{1\), ... , \(M\}\) such that \(x^{(n)}, y^{(m)} \in T\), we can use \textbf{Lemma C} for the set \(W \vcentcolon= \{x^{(n)}\), ... , \(x^{(N)}\), \(w\), \(y^{(M)}\), ... , \(y^{(m)}\}\) and conclude that \(\mathscr{B}(C, T, \vec{p}, v, w)\) holds. Therefore, we may assume that \(\{x^{(1)}\text{, ... , }x^{(N)}\}\text{ } \cap \text{ }T = \emptyset\) \text{ }or \(\{y^{(1)}\text{, ... , }y^{(M)}\} \text{ }\cap\text{ } T= \emptyset\). W.l.o.g., assume the former. Then, for fixed \(i, j \in \{0,1\}\), the only way that any of the edges \(v_ix_i^{(1)}\), \text{ }\(x_i^{(1)}x_i^{(2)}\), ... , \(x_i^{(N-1)}x_i^{(N)}\), \text{ }\(x_i^{(N)}w_i\)\text{ } can contribute to a self-avoiding path in \(B^T(C)\) from \(v_0\) to \(w_j\) is if all of them are present at once. Thus, we may think of these edges as just one single big edge between \(v_i\) and \(w_i\) which, by inheritance, is still present independently of all the other edges with probability \(p_{vw} \vcentcolon= p_{vx^{(1)}}p_{x^{(1)}x^{(2)}}\cdot\cdot\cdot p_{x^{(N-1)}x^{(N)}}p_{x^{(N)}w}\). Formally, if \(\widetilde{C}\) denotes the cycle constructed by deleting the vertices \(x^{(1)}\), ... , \(x^{(N)}\) and all edges incident to them from \(C\) and then adding the edge \(e^\ast \vcentcolon=vw\), then \textbf{Proposition 6} again shows that \(\mathscr{B}(\widetilde{C}\setminus e^\ast, T, (p_e)_{e\in  E(\widetilde{C})\setminus\{e^\ast\}}, v, w)\) holds and \textbf{Lemma B} shows that \(\mathscr{B}(\widetilde{C}, T, (p_e)_{e\in E(\widetilde{C})}, v, w)\) holds. Therefore, if the random graph \(B^T(\widetilde{C}) \vcentcolon= B^T(\widetilde{C}, (p_e)_{e\in E(\widetilde{C})})\) is constructed by letting horizontal copies of edges \(e \in E(\widetilde{C})\setminus \{e^\ast\}\) be present if and only if they are present in \(B^T(C)\), letting \(e^\ast_i =v_iw_i\) be present if and only if \(v_ix_i^{(1)}\), \text{ }\(x_i^{(1)}x_i^{(2)}\), ... , \(x_i^{(N-1)}x_i^{(N)}\), \text{ }\(x_i^{(N)}w_i\) are all present in \(B^T(C)\) and letting posts \(u_0u_1\) be present if and only if \(u \in T\) \(\bigl(\subseteq V(\widetilde{C})\bigr)\), we have \begin{center}
\(\mathbb{P}\bigl( v_0 \overset{B^T(C)}{\longleftrightarrow} w_0\bigr) \text{ }\text{ }=\text{ }\text{ }\mathbb{P}\bigl( v_0 \overset{B^T(\widetilde{C})}{\longleftrightarrow} w_0\bigr) \text{ }\text{ } \geq \text{ }\text{ }\mathbb{P}\bigl( v_0 \overset{B^T(\widetilde{C})}{\longleftrightarrow} w_1\bigr)\text{ }\text{ }=\text{ }\text{ } \mathbb{P}\bigl( v_0 \overset{B^T(C)}{\longleftrightarrow} w_1\bigr)\)
\end{center}\vspace{3pt}
by the coupling argument from above, concluding the proof.\end{proof}
We have found some first non-trivial graphs that we can apply our glueing technique to, and by glueing together copies of \(K_2\) and cycles \(C_n\) for (possibly different) \(n \in \mathbb{N}\), \(n\geq3\), we exactly construct a class of graphs that is already studied in other fields of graph theory:\vspace{2pt}
\begin{prop}
Any connected graph \(G\) is a cactus graph if and only if for every block \(H \in \mathbb{B}(G)\), either \(H \cong K_2\) or \(H \cong C_n\) for some \(n = n(H)\geq 3\).\vspace{6pt} 
\end{prop}
\begin{proof}
"\(\Leftarrow\)\text{": } (Contraposition.) Let \(G\) be connected and not a cactus. Then, there must exist two cycles \(C \neq \widetilde{C}\) in \(G\) that overlap in at least two different vertices \(x \neq y\). The graph \(H\) given by \(V(H) \vcentcolon=V(C)\) \(\cup\) \(V(\widetilde{C})\) and \(E(H)\vcentcolon= E(C)\) \(\cup\) \(E(\widetilde{C})\) is biconnected and therefore a subgraph of a block \(\widetilde{H}\) of \(G\). Also, since \(C \neq \widetilde{C}\), \(x\) can be chosen to have degree \(\geq 3\) in \(H\). In particular, \(x\) has degree \(\geq 3\) in \(\widetilde{H}\), so \(\widetilde{H}\) cannot be isomorphic to either \(K_2\) or any \(C_n\).\pagebreak\\
"\(\Rightarrow\)\text{": } Consider a cactus \(G\) and a block \(H\) of \(G\), \(\vert V(H)\vert \geq 2\). If \(\vert V(H)\vert = 2\), then \(H\) is isomorphic to \(K_2\). If \(\vert V(H)\vert \geq 3\), then it is easy to see that every vertex \(x \in V(H)\) must have atleast two neighbours, because if it had only one neighbour \(y\), then \(x\) would be isolated in \(H\setminus y\), and since \(\vert V(H\setminus y)\vert \geq 2\), \(H\setminus y\) would be disconnected, which is not possible since \(H\) is biconnected. Now, consider an arbitrary edge \(e =vw\in E(H)\). By the argument above, \(v\) must have another neighbour \(u \in V(H)\setminus\{v,w\}\). Since \(H\setminus v\) is still connected, there exists a self-avoiding path from \(u\) to \(w\) in \(H\setminus v\), and together with the edges \(uv\) and \(vw\) this path forms a cycle containing \(e = vw\). Thus, every edge of \(H\) is contained in a cycle in \(H\), and therefore, every vertex of \(H\) is contained in a cycle in \(H\) as well, because \(H\) is connected and has several vertices. Hence, for \(m \in \mathbb{N}\) different cycles \(C^{(1)}\), ... , \(C^{(m)}\) in \(H\), we have \(V(H) = V(C^{(1)})\) \(\cup\) ... \(\cup\) \(V(C^{(m)})\) and \(E(H) = E(C^{(1)})\) \(\cup\) ... \(\cup\) \(E(C^{(m)})\). We want to show that \(m = 1\). Suppose \(m\geq2\). Since \(H\) is connected, there must exist another cycle \(C^{(k)}\), \(k\geq 2\), such that \(C^{(1)}\) and \(C^{(k)}\) overlap in a common vertex \(z \in V(C^{(1)}) \cap V(C^{(k)})\). Since \(G\) is a cactus, \(C^{(1)}\) and \(C^{(k)}\) share no vertex other than \(z\), but since \(H\setminus z\) is still connected, we know that there must exist a self-avoiding path \(\eta\) in \(H\) that only leads across edges \(e \in E(H)\setminus(E(C^{(1)})\cup E(C^{(k)}))\) and connects a vertex \(z' \neq z\) from \(C^{(1)}\) with a vertex \(z''\neq z\) from \(C^{(k)}\). If we combine this path with a self-avoiding path \(\gamma\) from \(z\) to \(z'\) in \(C^{(1)}\) and a self-avoiding path \(\xi\) from \(z''\) to \(z\) in \(C^{(k)}\), then \(\gamma\) \(\circ\) \(\eta\) \(\circ\) \(\xi\) forms a cycle that shares the two vertices \(z\) and \(z'\) with \(C^{(1)}\) but is not identical to \(C^{(1)}\), because it contains \(z''\), contradicting the assumption that \(G\) is a cactus. Hence, we conclude \(m = 1\) and as such, \(H = C^{(1)}\).
\end{proof}\vspace{1pt}
\begin{proof}(Of \textbf{Theorem 2}.)\vspace{1pt}\\
Let \(G\) be a cactus graph. By \textbf{Theorem 1}, we need to show that any \(H \in\mathbb{B}(G)\) satisfies the strong bunkbed conjecture. By \textbf{Proposition 8}, we have either \(H \cong K_2\) or \(H\cong C_n\) for some \(n\geq 3\). By \textbf{Example 5} or \textbf{Proposition 7}, \(H\) satisfies the strong bunkbed conjecture.
\end{proof}\vspace{6pt}
Out of all the graphs that are not cactus graphs, the \emph{diamond graph} \(D\) is the smallest one in the sense that it has both the fewest vertices as well as the fewest edges. (We shall assign a deeper meaning to this notion of \(D\) being the "smallest" such graph later in this section.)\vspace{4pt}
\begin{center}
\includegraphics[scale=0.15]{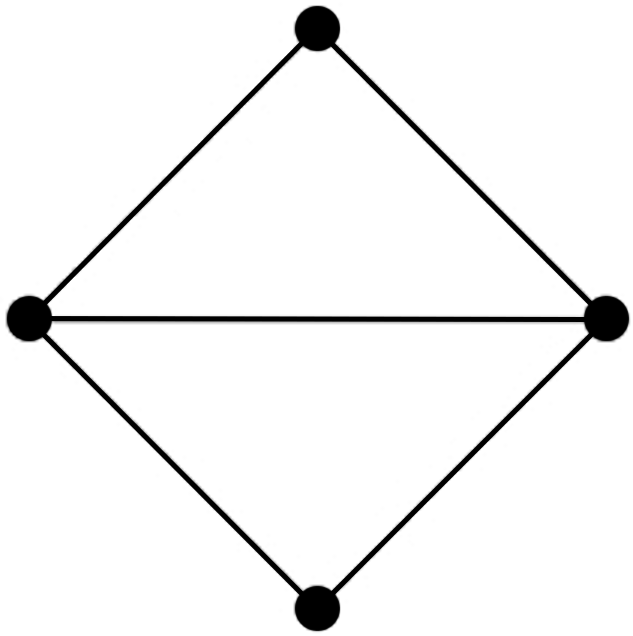}\vspace{6pt}\\
\emph{Fig. 4:\text{ } The diamond graph}\vspace{6pt}
\end{center}
\begin{prop}
The diamond graph satisfies the strong bunkbed conjecture.\vspace{6pt}
\end{prop}
\begin{proof}
Let \(T \subseteq V(D)\), \(\vec{p} = (p_e)_{e\in E(D)}\in [0,1]^{E(D)}\) and \(v,w \in V(D)\). W.l.o.g., assume \(v \neq w\). Note that \(D\) has two vertices of degree \(2\) and two vertices of degree \(3\).\vspace{4pt}\\
If at least one of the two vertices \(v\) and \(w\) is of degree \(3\), then \(e^\ast \vcentcolon= vw \in E(D)\) and \(D\setminus e^\ast\) is either isomorphic to the cycle graph \(C_4\) (if both \(v\) and \(w\) are of degree \(3\)) or it is isomorphic to \(K_2\) and \(K_3\) glued together at a single vertex (if one of the vertices \(v\) and \(w\) is of degree \(2\) and the other is of degree \(3\)). Either way, \(D\setminus e^\ast\) is a cactus, so \textbf{Theorem 2} shows that \(\mathscr{B}(D\setminus e^\ast, T, (p_e)_{e\in  E(D)\setminus\{e^\ast\}}, v, w)\) holds, and \textbf{Lemma B} shows that \(\mathscr{B}(D, T, \vec{p}, v, w)\) holds.\vspace{4pt}\\
If both \(v\) and \(w\) are of degree \(2\), then \(W \vcentcolon= V(D)\setminus\{v\}\) has the border \(\partial_DW = V(D)\setminus\{v,w\}\). By using \textbf{Lemma A}, assume that \(v, w \notin T\) and \(\vert T\vert \geq 2\), and since \(\vert V(D)\vert = 4\), the only choice left is \(T = V(D)\setminus\{v,w\} = \partial_DW\). \textbf{Lemma C} shows that \(\mathscr{B}(D, T, \vec{p}, v, w)\) holds.
\end{proof}
\begin{prop}
The complete graph \(K_4\) satisfies the strong bunkbed conjecture.\\
In particular, all graphs with at most four vertices satisfy the strong bunkbed conjecture.\vspace{6pt}
\end{prop}
\begin{proof}
Let \(T \subseteq V(K_4)\), \(\vec{p} = (p_e)_{e\in E(K_4)}\in [0,1]^{E(K_4)}\) and \(v,w \in V(K_4)\), w.l.o.g. \(v\neq w\). Since \(K_4\) is complete, we have \(e^\ast \vcentcolon= vw \in E(K_4)\) again. Since \(K_4\setminus e^\ast \cong D\), we apply \textbf{Proposition 9} to show that \(\mathscr{B}(K_4\setminus e^\ast, T, (p_e)_{e\in  E(K_4)\setminus\{e^\ast\}}, v, w)\) holds and \textbf{Lemma B} to show that \(\mathscr{B}(K_4, T, \vec{p}, v, w)\) holds.
\end{proof}
This yields a stronger version of \textbf{Theorem 2}:
\begin{thm}
Let \(G\) be a graph such that every biconnected component of \(G\) is either a cycle graph or has at most four vertices.\\
\emph{(}I.e. \(\forall H\in\mathbb{B}(G)\)\emph{:} \(H \cong K_2\) or \(H\cong D\) or \(H\cong K_4\) or \(H\cong C_n\) for some \(n = n(H) \geq 3\).\emph{)}\\
Then \(G\) satisfies the strong bunkbed conjecture.\vspace{6pt}
\end{thm}
\begin{proof}
This follows from \textbf{Theorem 1}, \textbf{Proposition 7} and \textbf{Proposition 10}.
\end{proof}
\noindent While we focused on the strong version of the bunkbed conjecture so far, \textbf{Theorem 1} also works for the weak version, which was considered by other authors before:
\begin{thm}
All complete graphs, edge differences of a complete graph and a complete subgraph, complete bipartite graphs and symmetric complete \(k\)-partite graphs satisfy the weak bunkbed conjecture.\vspace{6pt}
\end{thm}
For proofs we refer the interested reader to \cite{van2019bunkbed} and \cite{richthammer2022bunkbed}.\vspace{1pt}
\begin{proof}(Of \textbf{Theorem 3}.)\vspace{1pt}\\
This follows from \textbf{Theorem 1}, \textbf{Proposition 7} and \textbf{Theorem 12}.\end{proof}\vspace{2pt}
\noindent We have proved the two versions of the bunkbed conjecture for all graphs with somewhat "nice\text{" }blocks, where the meaning of that word depends on the version of the conjecture. \textbf{Theorem 11} states that all graphs whose blocks are either cycles or have at most four vertices satisfy the strong bunkbed conjecture, and connected such graphs are exactly cactus graphs where any of the missing chords may have been added to cycles \(C \cong C_4\). We will describe similar classes of graphs in a different way in a moment.\vspace{4pt}\\
Note that if we can show that the complete graph \(K_n\) satisfies the strong bunkbed conjecture for some \(n \in \{5\), ... , \(7221\}\), we may replace the words "at most four" in \textbf{Theorem 11} by "at most \(n\)", and similarly we could then add biconnected graphs with at most \(n\) vertices to the list of allowed blocks in \textbf{Theorem 3}. Also note that this is impossible for \(n \geq 7222\), because the counterexample to the bunkbed conjecture given in \cite{gladkov2024bunkbed} has \(7222\) vertices, so in particular, \(K_{7222}\) (and all complete graphs \(K_n\) for \(n \geq 7223\)) must not satisfy the strong bunkbed conjecture.\vspace{4pt}\\
Finally, we prove restraints a graph must succumb to in order to have a chance to violate the strong bunkbed conjecture and demonstrate how stronger versions of \textbf{Proposition 10} for bigger complete graphs \(K_n\), \(n\geq 5\), might strengthen these restraints in the future.
\begin{proof}(Of \textbf{Theorem 4}.)\vspace{1pt}\\
By \textbf{Theorem 11}, any counterexample to the strong bunkbed conjecture has a biconnected component with more than four vertices that is not a cycle. It is easy to see that any biconnected graph with more than four vertices that is not a cycle must contain one of the two graphs given in \textbf{Theorem 4} as a minor, immediately yielding \textbf{Theorem 4}.\end{proof}
\noindent We now show how the two graphs from \textbf{Theorem 4} exactly characterize the class of graphs whose blocks are either cycles or have at most four vertices, and that similar restraints on counterexamples to the strong bunkbed conjecture might indeed be proved likewise for other cases than \(n=4\). We want to apply the so called \emph{graph minor theorem} that was infamously proved over a course of 20 papers; see for example \cite{robertson2004graph}. Given a graph \(G\), any graph \(H\) (up to graph isomorphism) that is constructed by repeatedly deleting and/or contracting some sets of edges of \(G\) and then deleting a subset of the arising isolated vertices is called a \emph{minor} of \(G\). Note that we consider every graph \(G\) to be a minor of itself and that the minor relation is both transitive and anti-symmetric (up to graph isomorphism).\vspace{4pt}\\
Now, consider a class\footnote{As we only care about graphs up to graph isomorphism, we may w.l.o.g. view the vertices of any graph \(G\) as embedded in the natural numbers, making \(\mathbb{G}\) an actual set instead of a proper class.} \(\mathbb{G}\) of graphs. \(\mathbb{G}\) is said to be \emph{minor-closed} if and only if for every graph \(G \in \mathbb{G}\) and every minor \(H\) of \(G\): \(H \in \mathbb{G}\). A possible way of stating the graph minor theorem is that if \(\mathbb{G}\) is minor-closed, then it is characterized by a finite set \(\mathcal{H}\) of graphs such that for any given graph \(G\): \(G\in\mathbb{G}\) if and only if no \(H\in\mathcal{H}\) is a minor of \(G\). Typically, we choose \(\mathcal{H}\) to be a smallest set with this property. Then, the elements of \(\mathcal{H}\) are called the \emph{forbidden minors} of \(\mathbb{G}\) (up to graph isomorphism).\vspace{4pt}\\
For \(n\in\mathbb{N}\), \(n \geq 2\), let \(\mathbb{G}_n\) denote the class of all graphs \(G\) such that every block \(H\in \mathbb{B}(G)\) is either a cycle or has at most \(n\) vertices. For example, \(\mathbb{G}_2 = \mathbb{G}_3\) is the class of graphs whose connected components are cactus graphs, and \(\mathbb{G}_4\) is the class of graphs whose connected components are cactus graphs where some additional chords may have been added to cycles \(C\cong C_4\); i.e. the class of graphs discussed in \textbf{Theorem 11}.\vspace{4pt}\\
Since deleting or contracting an edge from a block can never make the block bigger or transform a block that is a cycle into one or more non-cycle and non-trivial blocks, \(\mathbb{G}_n\) is minor-closed for \(n\geq 2\). By the graph minor theorem, \(\mathbb{G}_n\) is characterized by a set \(\mathcal{H}_n\) of forbidden minors, where for instance \(\mathcal{H}_2 = \mathcal{H}_3 = \{D\}\). (See \cite{el1988complexity}.) Note that \(\mathcal{H}_n\) can be given explicitly for any explicit \(n\), although this will get complicated very fast.\vspace{4pt}\\
If we can show that the complete graph \(K_n\) satisfies the strong version of the bunkbed conjecture for some \(n\geq 2\), then by \textbf{Theorem 1} and \textbf{Proposition 7}, all graphs \(G \in \mathbb{G}_n\) satisfy the strong bunkbed conjecture, and as such, every counterexample to the strong bunkbed conjecture must contain some \(H\in\mathcal{H}_n\) as a minor. The case \(n = 4\) was dealt with in \textbf{Proposition 10}, and the forbidden minors are indeed given by (up to graph isomorphism)\vspace{1pt}
\begin{center}\includegraphics[scale=0.055]{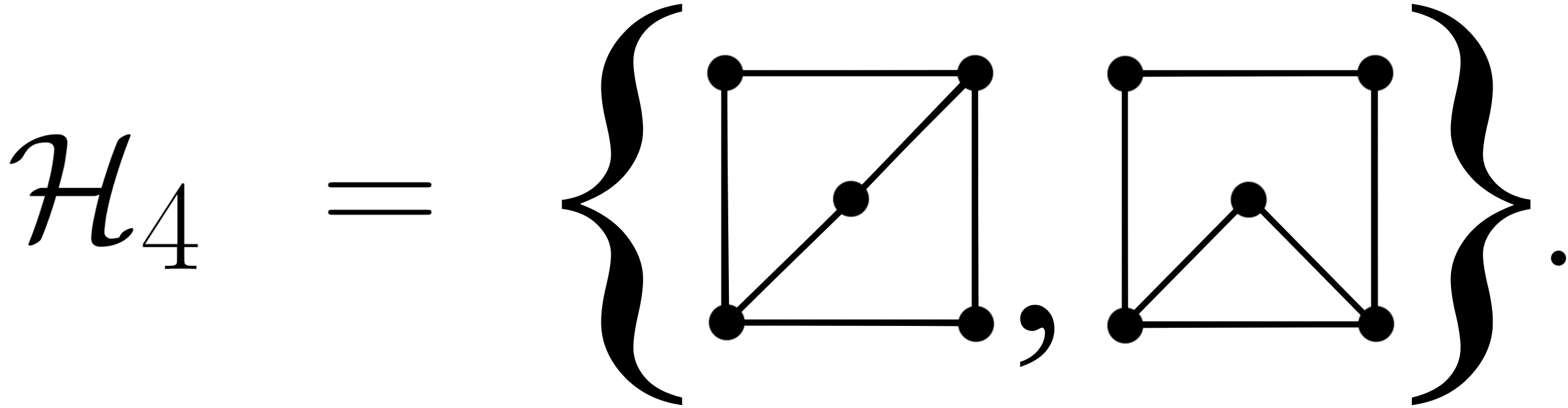}
\end{center}
As for a quick sketch of a proof of this, the reader may flesh out the following arguments:\vspace{4pt}\\
"\(\supseteq\)\text{":} Both graphs above are not elements of \(\mathbb{G}_4\) and therefore contain an element of \(\mathcal{H}_4\) as a minor, but the only such minors that are not elements of \(\mathbb{G}_4\) are the two graphs themselves.\vspace{4pt}\\
"\(\subseteq\)\text{":} Verify that every biconnected graph with at least five vertices that is not a cycle must contain one of the two graphs above as a minor. In particular, every \(H \in \mathcal{H}_4\) has a block that contains an \(H'\) as a minor where \(H'\) is one of the two graphs above, since \(H \notin \mathbb{G}_4\). Then, as \(H' \notin \mathbb{G}_4\), \(H'\) contains another \(H'' \in \mathcal{H}_4\) as a minor. By transitivity of the minor relation, \(H''\) is a minor of \(H\), and since \(\mathcal{H}_4\) is chosen to be minimal, it follows that \(H \cong H''\). (Otherwise, \(H\) would not be needed as a forbidden minor to characterize \(\mathbb{G}_4\).) Hence \(H\) is also a minor of \(H'\), and by anti-symmetry of the minor relation it follows that \(H \cong H'\).\vspace{4pt}\\
This reasoning shows that the assumptions about \(G\) in \textbf{Theorem 11} are precisely equivalent to \(G\) not containing any of the two minors \(H\in\mathcal{H}_4\) given in \textbf{Theorem 4}.\vspace{4pt}\\
We can also use this philosophy to conclude:\vspace{10pt}
\begin{thm}
There is a smallest, finite set \(\mathcal{H} \neq \emptyset\) of graphs such that any given graph satisfies the strong bunkbed conjecture if and only if it contains no element of \(\mathcal{H}\) as a minor. In particular, with respect to the minor relation, \(\mathcal{H}\) consists exactly of the smallest counterexamples to the strong bunkbed conjecture \emph{(}up to graph isomorphism\emph{)} and every \(H \in \mathcal{H}\) is biconnected and contains a non-trivial subdivision of \(D\) as a minor.\vspace{6pt}
\end{thm}
\begin{proof}
Let \(\mathbb{G}\) denote the class of graphs that satisfy the strong bunkbed conjecture and let \(G\in \mathbb{G}\). It is easy to see that, as far as the conjecture is concerned, assigning certain edges \(e \in E(G)\) an edge weight of \(p_e = 0\) or \(p_e = 1\) corresponds to deleting or contracting these edges, respectively. Therefore, \(\mathbb{G}\) is minor-closed. (This is why a-priori, this theorem only works for the strong version of the conjecture.)  Now, the set \(\mathcal{H}\) is the set of forbidden minors of \(\mathbb{G}\) given by the graph minor theorem. Also, since the bunkbed conjecture is false, it follows that \(\mathcal{H} \neq \emptyset\), and some explicit element \(H \in \mathcal{H}\) can be found by searching the minors of the counterexample to the bunkbed conjecture given in \cite{gladkov2024bunkbed}. Obviously, since the elements of \(\mathcal{H}\) are the smallest counterexamples to the strong bunkbed conjecture, any \(H\in \mathcal{H}\) has to be connected. Moreover, if any \(H\in \mathcal{H}\) was not biconnected, there would exist an \(x \in V(H)\) such that \(H\setminus x\) is disconnected. In particular, for some \(m\geq 2\), there would exist a decomposition \(H = H_1\) \(\cup\)\text{ ... }\(\cup\) \(H_m\) such that \(H_1\)\text{, ... , }\(H_m\) are graphs with \(\vert V(H_i) \vert \geq 2\) that overlap only at the vertex \(x\), and as such, \(\vert V(H_i)\vert \leq \vert V(H) \vert -1\). Since \(H\) is a smallest counterexample, \(H_1\)\text{, ... , }\(H_m\) would have to satisfy the strong bunkbed conjecture, but by \textbf{Theorem 1b}, \(H\) could then not be a counterexample either. By \textbf{Theorem 4}, \(H\) must contain one of the two smallest non-trivial subdivisions of \(D\) as a minor.
\end{proof}
\printbibliography
\end{document}